\documentclass[10pt, reqno]{amsart}

\usepackage{enumerate}

\numberwithin{equation}{section}

\newtheorem{theorem}{Theorem}[section]
\newtheorem*{theorem*}{Theorem}	

\newtheorem{corollary}[theorem]{Corollary}
\newtheorem{lemma}[theorem]{Lemma}

\renewcommand{\Re}{\operatorname{Re}}

\newcommand{\bb}[1]{\left(#1\right)}
\renewcommand{\mod }{\operatorname{mod}}
\renewcommand{\epsilon }{\operatorname{\varepsilon}}
\renewcommand{\rho }{\operatorname{\varrho}}
\renewcommand{\theta }{\operatorname{\vartheta}}

\def\beq{\begin{equation}}
\def\eeq{\end{equation}}

\title{The exceptional set of Goldbach problem and Linnik's constant}

\author{Genheng Zhao}
\email{zhaogenheng@amss.ac.cn}

\begin{document}
\maketitle

\begin{abstract}
Let $E(X)$ denote the number of even integers below $X$ which are not a sum of two primes. We prove the bound $E(X)=O(X^{\frac{7}{10}})$, where the implicit constant is ineffective. The method applied here also leads to $P(q)=O(q^5)$, where $P(q)$ denotes the least prime, if it exists, in any arithmetic progression modulo $q$.
\end{abstract}

\section{Introduction}
The famous Goldbach conjecture asserts that every even integer $m \geq 6$ is a sum of two odd primes. Though the original conjecture is still unsolved today, there have appeared various approximations which are available by current methods. One typical example is to bound the number of possible exceptions.

Let $E(X)$ denote the size of exceptional set, which consists of even integers $m\leq X$ which are not a sum of two odd primes. In the 1920s, under the generalized Riemann hypothesis, Hardy-Littlewood \cite{Hardy-Littlewood} used their newly invented circle method to show the bound
\begin{equation}\label{eq:HardyLittlewood}
E(X)=O_{\epsilon}(X^{1/2+\epsilon}),\quad \forall \epsilon>0.
\end{equation}
The later breakthrough of Vinogradov \cite{Vinogradov} in the 1930s was able to reach the unconditional bound
\begin{equation}\label{eq:Vinogradov}
E(X)=O_A(X(\log X)^{-A}), \quad \forall A>0.
\end{equation}
Bounds of such type are often related to Siegel-Walfisz theorem, which gives the uniform distribution of primes $p\leq X$ among congruence classes with moduli $q\leq (\log X)^A$, $\forall A>0$. When $q$ gets larger, such  uniformity will  inevitably be broken by the possibly existing Siegel zero of some $L$-function. Hence for better bounds of $E(X)$, the effect of Siegel zero must be taken into account. This was done by Vaughan in 1972, allowing him to replace Siegel-Walfisz theorem by Page's theorem to obtain
\begin{equation}\label{eq:Vaughan}
E(X)=O(Xe^{-c\sqrt{\log X}}).
\end{equation}
Soon after this improvement, in 1975 Montgomery-Vaughan \cite{Montgomery-Vaughan} showed that
\begin{equation}\label{eq:MontgomeryVaughan}
E(X)=O(X^{1-\delta})
\end{equation}
holds with some $\delta>0$. Their method  can be seen as a generalization of Linnik's method \cite{Linnik} showing
\begin{equation}\label{eq:Linnik}
P(q)=O(q^L)
\end{equation}
holds with some $L<\infty$, where $P(q)$ denotes the least prime in any arithmetical progression $\{a,q+a,2q+a,\cdots\}$ with $1\leq a\leq q$ and $(a,q)=1$.

Linnik's method relies on three intricate principles on the distribution of zeroes of $L$-functions near the line $\sigma=1$, including a zero-free region with at most one exception, a log-free zero density estimate and a quantified version of Deuring-Heilbronn phenomenon. Roughly speaking, these three principles together serve as a substitute for the zero free region
\begin{equation}\label{eq:ZeroFreeRegion}
\sigma \geq 1-\frac{c}{\log (qT)},\quad |t|\leq T
\end{equation}
based on which the bound \eqref{eq:Linnik} would be quite clear. Obviously, the admissible values of $\delta$ and $L$ are closely related to the constants appeared in these principles. For such a result concerning $L$, see \cite[Corollary 18.8]{Iwaniec-Kowalski}.

The first admissible value of $\delta$ was obtained by J. R. Chen and J. M. Liu \cite{Chen-Liu}, where they showed $\delta=0.05$. They \cite{Chen-Liu-2} also showed that $L=13.5$ is admissible. In 1991, Heath-Brown \cite{Heath-Brown} developed the currently sharpest tool to obtain zero free regions of $L$-functions, giving the improvement $L=5.5$. This was further improved to $L=5.2$ by Xylouris \cite{Xylouris} in 2009, which is still best up to date. On the other hand, by exploiting Heath-Brown's idea, the admissible value of $\delta$ was improved to $\delta=0.086$ by H. Z. Li \cite{Li}, $\delta=0.121$ by W. C. Lu \cite{Lu} and finally $\delta=0.28$ by Pintz \cite{Pintz2018}, which seems to match the result $L=5.2$.

In this paper, with a refinement of Pintz's method, we verify that $\delta=0.3$ is admissible. In other words, we have the following.

\begin{theorem}\label{thm:main1}
$E(X)=O(X^{\frac{7}{10}})$, where the implicit constant is ineffective.
\end{theorem}

In light of Pintz's reduction \cite[Section 2]{Pintz2018}, we only need to consider the contribution of zeroes of $L$-functions for a single moduli $q$, in a highly restricted area
\begin{equation}\label{eq:RestrictedArea}
\sigma \in [1-H(\log q)^{-1},1],\quad |t|\leq H,
\end{equation}
where $H$ is a fixed positive constant. We may denote the multi-set of these zeroes by $\mathcal{Z}(\chi, H)$, $\mathcal{Z}(\mathcal{K}, H)$ the union of $\mathcal{Z}(\chi , H)$ for all $\chi \in \mathcal{K}$ and  always write
\begin{equation}\label{eq:rho}
\rho=\beta+i\gamma,\quad 1-\rho= \frac{\lambda +i\mu }{\log q}.
\end{equation}
It remains to show the following result.

\begin{theorem}\label{thm:main2}
Let $q,H\geq 1$ and $c_0>0$. Let $\{\mathcal{K}_i\}_{i\geq 1}$ be pairwise disjoint sets of characters module $q$. For $i\geq1$, let $\mathcal{Z}_i=\mathcal{Z}(\mathcal{K}_i,H)$ and suppose
\begin{equation}\label{eq:Condition}
\max_{(\chi,\chi')\in \mathcal{K}_i^2}\operatorname{cond}(\overline{\chi} \chi')\leq c_0^{-1},\quad \min_{\rho \in \mathcal{Z}_i} \lambda \geq c_0.
\end{equation}
Then for $q\geq q(H,c_0)$ we have
\begin{equation}\label{eq:final}
\sum_{i\geq 1} \bb{\sum_{\rho \in \mathcal{Z}_i} e^{-\frac{10}{3} \lambda}}^2 \leq 1-c_1,
\end{equation}
where $c_1>0$ depends only on $c_0$.
\end{theorem}

In \cite{Pintz2018}, Pintz showed \eqref{eq:final} with $25/7$ in place of $10/3$. Hence we have reduced the value  by $5/21= 0.2380\cdots$. We can deduce from   the original argument  the estimate \eqref{eq:final} with $1.277$ in place of $1-c_1$.  The small  saving $0.277$ consists of  two parts: a refined estimate of sums of the type
\begin{equation}\label{eq:SumType1}
\sum_{\rho \in \mathcal{Z}}e^{-\frac{10}{3}\max \{\lambda,\Lambda \}}
\end{equation}
which contributes 0.114 and an estimate of sums of the type
\begin{equation}\label{eq:SumType2}
\sum_{\rho \in \mathcal{Z}_i}(e^{-\frac{10}{3}\lambda }-e^{-\frac{10}{3}\Lambda})
\end{equation}
which contributes 0.163.

We shall mention that the refined argument here has also uncovered the potential of Pintz's method to improve some other results such as Linnik's constant. Let $\mathcal{Z}(q,H)$ the union of all $\mathcal{Z}(\chi,H)$ with $\chi$ modulo $q$.

\begin{theorem}\label{thm:main3}
Let $q,H\geq 1$, $c_0>0$. Assume that
\begin{equation}\label{eq:MinCondition}
\min_{\rho \in \mathcal{Z}(q,H)} \lambda \geq c_0.
\end{equation}
Then for $q\geq q(H,c_0)$ we have
\begin{equation}\label{eq:final2}
\sum_{\rho\in \mathcal{Z}(q,H) } e^{-5\lambda} \leq 1-c_1,
\end{equation}
where $c_1>0$ depends only on $c_0$.
\end{theorem}

 Since it is much simpler than Theorem \ref{thm:main2}, we are not going to provide details for its proof. Recall the standard notation
\beq
\pi(X;q,a)=\sum_{\substack{p\leq X\\p\equiv a(\mod q)}} 1.
\eeq 
 Theorem \ref{thm:main3} combined with the argument in  \cite{Maynard} gives the following result.

\begin{theorem}\label{thm:main4}
Let $1\leq a\leq q$ and $(a,q)=1$. Let $X=q^{\theta}$ with $\theta\geq 5$ being fixed. Then we have for any $\epsilon>0$ that
\begin{equation}\label{eq:Bound1}
0<\pi(X;q,a)\leq \frac{(2+\epsilon)X}{\phi(q)\log X}, \quad q\geq q(\epsilon).
\end{equation}
\end{theorem}

This clearly shows $L=5$ is addmissible. The improvement from $L=5.2$ to $L=5$ is essentially due to Pintz's generalization of Heath-Brown's new zero density estimate, which counts all zeroes of $L(s,\chi )$ rather than the number of $L(s,\chi )$ which has at least a zero in the area
\begin{equation}\label{eq:Area}
\sigma \in [1-\Lambda(\log q)^{-1},1],\quad |t| \leq H,
\end{equation}
with $\Lambda\leq 1.6$. This fully avoids the use of the function \eqref{eq:f1} which is less optimal than the function \eqref{eq:F2}, in Heath-Brown's  method.

The remaining parts of this paper are arranged as follows. In Section \ref{sec:Preliminary} we will introduce some preliminary results for the estimates of certain sums over zeroes in Section \ref{sec:WeightedSums}. Finally   we  prove Theorem \ref{thm:main2} in Section \ref{sec:ProofTheorem2}. 

\section{Preliminary results}\label{sec:Preliminary}

Let $H\geq 1$ and $z_0,t_0\in [0,2H]$. We now introduce several special functions used in \cite{Heath-Brown}, including 
\begin{equation}\label{eq:f1}
f_1(t)=1_{t \in [0,t_0]} \sinh(z_0 (t_0-t))
\end{equation}
and its Laplace transform
\begin{equation}\label{eq:F1}
F_1(z)=\int_{0}^{\infty}f_1(t)e^{-tu}dt=\frac{1}{2}\bb{\frac{e^{z_0t_0}}{z_0+z}+\frac{e^{-z_0t_0}}{z_0-z}-\frac{2z_0e^{-zt_0}}{z_0^2-z^2}}.
\end{equation}
Note first that $f_1(t)$ and $F_1(z)$ vary in a compact set depending merely on $H$. Moreover, the function $F_1(z)$ obeys the property that for $\Re(z)\geq 0$,
\begin{equation}\label{eq:F1Property}
\Re(F_1(z))\geq \frac{z_0}{2} e^{z_0 t_0}|F_2(z_0+z)|,
\end{equation}
where
\begin{equation}\label{eq:F2}
F_2(z)=\bb{\frac{1-e^{-t_0 z}}{z}}^2.
\end{equation}
We will use this relation to treat
\begin{equation}\label{eq:ToTreat}
\max_{\rho \in \mathcal{Z}(\chi,H)}\sum_{\rho' \in \mathcal{Z}(\chi, H)} |F_2(\lambda+\lambda'-i(\mu-\mu'))|
\end{equation}
for some $\chi$ modulo $q$. Note that
\begin{equation}\label{eq:F1Bound}
|F_1(z)|\ll \frac{1}{1+|z|}
\end{equation}
and by a zero density estimate of Jutila \cite[Theorem 2]{Jut}, when $q\geq q(H)$,
\begin{equation}\label{eq:ZeroDensity}
\# \mathcal{Z}(\chi,H)=O(e^{3H}),
\end{equation}
Let $\phi=1/3$ be fixed. Then \cite[Lemma 5.2]{Heath-Brown} and \cite[Lemma 5.3]{Heath-Brown} easily lead to the following estimate.

\begin{lemma}\label{lem:1}
Let $H\geq 1$, $\epsilon>0$ and suppose $s=\sigma+it$ satisfy
\begin{equation}\label{eq:sCondition}
|\sigma-1| \leq \frac{H}{\log q},\quad |t|\leq 2H.
\end{equation}
Let $\chi$ be a non-principal character modulo $q$. If $\chi$ is non-principal, then for $q\geq q(H,\epsilon)$ we have
\begin{equation}\label{eq:lem1Result}
\sum_{\rho \in \mathcal{Z}(\chi,H)} \Re F_1((s-\rho)\log q)\leq F_1((\sigma-1)\log q)+\frac{\phi}{2}f_1(0)+\epsilon,
\end{equation}
\end{lemma}

\begin{proof}
Since $\chi$ is non-principal and $f_1$ vary in a compact set depending only on $H$, by \cite[Lemma 5.2]{Heath-Brown} there exists $\delta=\delta(H,\epsilon)$ such that
\begin{equation}\label{eq:DeltaBound}
\sum_{|1+it-\rho|\leq \delta} \Re F_1((s-\rho)\log q)\leq -K_1(s,\chi)+\frac{\phi}{2}f_1(0)+\epsilon,
\end{equation}
where
\begin{equation}\label{eq:K1}
K_1(s,\chi)=\frac{1}{\log q}\sum_{n\geq 1}\Lambda(n)\Re\bb{\frac{\chi(n)}{n^s}}f_1\bb{\frac{\log n}{\log q}}.
\end{equation}
We now remove those zeroes with $\beta<1-H(\log q)^{-1}$ in \eqref{eq:DeltaBound}. This can be done since $\sigma \geq 1-H(\log q)^{-1}$ and hence by \eqref{eq:F1Property}, $\Re F_1((s-\rho))\geq 0$. Then, we can add those zeroes in $\mathcal{Z}(\chi,H)$ with $|1+it-\rho|\geq \delta$ into the reduced sum, since by \eqref{eq:ZeroDensity} their total number is $O(e^{3H})$ and by \eqref{eq:F1Bound} each contributes $O(\delta^{-1}(\log q)^{-1})$. Thus we obtain
\begin{equation}\label{eq:Intermediate}
\sum_{\rho \in \mathcal{Z}(\chi,H)} \Re F_1((s-\rho)\log q)\leq -K_1(s,\chi)+\frac{\phi}{2}f_1(0)+2\epsilon,
\end{equation}
provided $q\geq q(H,\epsilon)$. Meanwhile, by \cite[Lemma 5.3]{Heath-Brown} we have
\begin{equation}\label{eq:K1Bound}
-K_1(s,\chi)\leq K_1(\sigma,\chi_q^0)\leq F_1((\sigma-1)\log q)+\epsilon.
\end{equation}
Now the conclusion is clear if we change the value of $\epsilon$.
\end{proof}

We now obtain a bound of \eqref{eq:ToTreat}.

\begin{lemma}\label{lem:2}
Let $H\geq 1$ and $\epsilon>0$. Let $\chi$ be a non-principal character modulo $q$. Assume that each zero $\rho \in \mathcal{Z}(\chi,H)$ satisfies $\lambda\geq \lambda_0\geq 0$. Then for $q\geq q(H,\epsilon)$, we have
\begin{equation}\label{eq:lem2Result}
\max_{\rho \in \mathcal{Z}(\chi,H)} \sum_{\rho' \in \mathcal{Z}(\chi, H)} |F_2(\lambda +\lambda'-i(\mu-\mu'))| \leq B(t_0,\lambda,\lambda_0)+\epsilon,
\end{equation}
where
\begin{equation}\label{eq:B}
B(t_0,\lambda,\lambda_0)=\frac{\phi}{2}\bb{\frac{1-e^{-2(\lambda+\lambda_0)t_0}}{\lambda+\lambda_0}}+\frac{1-e^{-\lambda t_0}}{\lambda (\lambda+\lambda_0)}+\frac{e^{-2(\lambda+\lambda_0)t_0}-e^{-\lambda t_0}}{(\lambda+\lambda_0)(\lambda+2\lambda_0)}.
\end{equation}
\end{lemma}

\begin{proof}
Fix $\rho \in \mathcal{Z}(\chi,H)$. We then set $z_0=\lambda+\lambda_0$ and $(1-s_0)\log q= \lambda_0+i\mu$, which leads to
\begin{equation}\label{eq:Setting}
\begin{split}
\sum_{\rho' \in \mathcal{Z}(\chi, H)} |F_2(\lambda+\lambda'-i(\mu-\mu'))| \leq \frac{2}{z_0}e^{-z_0t_0} \sum_{\rho' \in \mathcal{Z}(\chi, H)} \Re F_1(\lambda'-\lambda_0-i(\mu-\mu')).
\end{split}
\end{equation}
By \eqref{eq:lem1Result}, for $q\geq q(H,\epsilon)$ we have
\begin{equation}\label{eq:SumBound}
\begin{split}
& \sum_{\rho' \in \mathcal{Z}(\chi, H)} \Re F_1(\lambda'-\lambda_0-i(\mu-\mu')) \\
=& \sum_{\rho' \in \mathcal{Z}(\chi, H)} \Re F_1((s_0-\rho')\log q) \\
\leq & f_1(0)+F_1(-\lambda_0)+2\epsilon \\
=&\frac{\phi}{2}\bb{\frac{e^{z_0t_0}-e^{-z_0t_0}}{2}}+\frac{1}{2}\bb{\frac{e^{z_0 t_0}}{z_0-\lambda_0}+\frac{e^{-z_0 t_0}}{z_0+\lambda_0}-\frac{2z_0 e^{2\lambda_0t_0}}{z_0^2-\lambda_0^2}}+2\epsilon,
\end{split}
\end{equation}
which immediately leads to \eqref{eq:lem2Result} from \eqref{eq:Setting}.
\end{proof}

We need the following elementary property of $B(t_0,\lambda,\lambda_0)$.

\begin{lemma}\label{lem:3}
Suppose $\lambda,\Lambda, t_0,\lambda_0 \geq 0$, then we have
\begin{equation}\label{eq:lem3Result}
e^{-\max\{\Lambda-\lambda,0\}t_0}B(t_0,\lambda,\lambda_0)\leq B(t_0,\Lambda,\lambda_0).
\end{equation}
\end{lemma}

\begin{proof}
It suffices to show $B(t_0,\lambda,\lambda_0)$ is decreasing while $e^{\lambda t_0}B(t_0,\lambda,\lambda_0)$ is increasing, with $\lambda$. When $\lambda_0=0$, we see
\begin{equation}\label{eq:B0}
B(t_0,\lambda,0)=\frac{\phi}{2}\bb{\frac{1-e^{-2\lambda t_0}}{\lambda}}+\bb{\frac{1-e^{-\lambda t_0}}{\lambda}}^2
\end{equation}
clearly obeys this property. Hence we may assume $\lambda_0>0$. Meanwhile, since
\begin{equation}\label{eq:FirstTerm}
\frac{\phi}{2}\bb{\frac{1-e^{-2(\lambda+\lambda_0)t_0}}{\lambda+\lambda_0}}
\end{equation}
satisfies this property, it remains to consider
\begin{equation}\label{eq:g}
g(\lambda)=\frac{1-e^{-\lambda t_0}}{\lambda (\lambda+\lambda_0)}+\frac{e^{-2(\lambda+\lambda_0) t_0}-e^{-\lambda t_0}}{(\lambda+\lambda_0)(\lambda+2\lambda_0)}.
\end{equation}
Notice that by Laplace inversion, we have
\begin{equation}\label{eq:LaplaceInversion}
g(\lambda)=\frac{1}{\lambda_0}\int_{0}^{2t_0}\bb{e^{-2\lambda_0 \max\{t-t_0,0\}}-e^{-\lambda_0t}}e^{-\lambda t}dt
\end{equation}
and hence
\begin{equation}\label{eq:gExp}
e^{\lambda t_0}g(\lambda)=\frac{1}{\lambda_0}\int_0^{t_0}(1-e^{-\lambda_0(t_0-t)})(e^{\lambda t}+e^{-(\lambda+2\lambda_0)t})dt.
\end{equation}
From this we derive that
\begin{equation}\label{eq:gDerivative}
g'(\lambda)=-\frac{1}{\lambda_0}\int_{0}^{2t_0}\bb{e^{-2\lambda_0 \max\{t-t_0,0\}}-e^{-\lambda_0t}}te^{-\lambda t}dt\leq 0
\end{equation}
and
\begin{equation}\label{eq:gExpDerivative}
(e^{\cdot t_0}g)'(\lambda)=\frac{1}{\lambda_0}\int_0^{t_0}(1-e^{-\lambda_0(t_0-t)})t(e^{\lambda t}-e^{-(\lambda+2\lambda_0)t})dt\geq 0,
\end{equation}
as desired.
\end{proof}

The final task of this section is to label the zeroes. In what follows we will keep the notations consistent with that in Theorem \ref{thm:main2}. For convenience we label $\rho \in \mathcal{Z}_i$ by
\begin{equation}\label{eq:Labeling}
\lambda_{i,1}\leq \lambda_{i,2}\leq \cdots,
\end{equation}
and $\{\mathcal{Z}_i\}_{i\geq 1}$ by
\begin{equation}\label{eq:LabelingSets}
\lambda_{1,1}\leq \lambda_{2,1}\leq \cdots.
\end{equation}
For $i\geq 1$, let
\begin{equation}\label{eq:Si}
S_i=\sum_{j\geq 1}e^{-\frac{10}{3}\lambda_{i,j}}.
\end{equation}
It now remains to show
\begin{equation}\label{eq:FinalEq}
\sum_{i\geq 1}S_i^2\leq 1-c_1.
\end{equation}

For each $\Lambda\leq H$, let $N(\Lambda)$ denote the number of zeroes $\rho=\beta+i\gamma$ of
\begin{equation}\label{eq:Mq}
\prod_{\chi (\mod q)}L(s,\chi),
\end{equation}
in the region
\begin{equation}\label{eq:Region}
\beta\geq 1-\frac{\Lambda}{\log q},\quad |\gamma| \leq H.
\end{equation}
The classical zero free region for $L$-functions implies that $N(\Lambda)\leq 1$ holds for some $\Lambda>0$. Meanwhile, if $N(\Lambda)=1$ does hold for some very small value of $\Lambda$, then $N(\Lambda)\leq 1$ will be expected to hold for some considerably large value of $\Lambda$, due to the Deuring-Heilbronn phenomenon. Explicit versions of this phenomenon have already been established in \cite{Heath-Brown} and \cite{Xylouris}, and they are still applicable in our setting if  $q\geq q(H,c_0)$, which we henceforth assume,

We end up this section with the following Lemma, according to which our final argument will be separated into several cases.

\begin{lemma}\label{lem:Cases}
Let $q$ be sufficiently large.
\begin{enumerate}[(1)]
\item If $\lambda_{1,1}\leq 0.01$, then $N(5.68)\leq 1$.
\item If $\lambda_{1,1}\leq 0.10$, then $N(3.08)\leq 1$.
\item If $\lambda_{1,1}\leq 0.30$, then $N(1.58)\leq 1$.
\item If $\lambda_{1,1}\leq 0.40$, then $N(1.29)\leq 1$.
\item If $\lambda_{1,1}\leq 0.60$, then $N(0.92)\leq 2$.
\item If $\lambda_{1,1}\leq 0.62$, then either $N(0.85)\leq 1$ or $N(0.91)\leq 2$.
\item If $\lambda_{1,1}\leq 0.64$, then $N(0.85)\leq 2$.
\item If $\lambda_{1,1}\leq 0.68$, then $N(0.74)\leq 2$.
\end{enumerate}
Despite the value of $\lambda_{1,1}$, we have always $N(1.09\log\lambda_{1,1}^{-1})\leq 1$, $N(0.702)\leq 2$, and $\lambda_{i,1}\geq 0.857$ for $i\geq 5$.
\end{lemma}

\begin{proof}
Values of $\Lambda$ for which $N(\Lambda)\leq 1$ are taken from Table 2,3,4,5,7 of \cite{Heath-Brown} and $N(\Lambda)\leq 2$ are taken from Tabellen 2,3,7 of \cite{Xylouris}. The range $\lambda\leq 1.09 \log \lambda_{1,1}^{-1}$ is a simple consequence of \cite[Lemma 8.8]{Heath-Brown}. The result $\lambda_{i,1}\geq 0.857$ for $i\geq 5$ is \cite[Theorem 3]{Heath-Brown}.
\end{proof}

\section{Weighted sums over zeroes}\label{sec:WeightedSums}

In this section we confine our attention to the sum
\begin{equation}\label{eq:S}
S=\sum_{j\geq 1}e^{-\frac{10}{3}\lambda_j},
\end{equation}
with $\max \{c_0,\lambda_0\}\leq \lambda_1\leq \lambda_2\leq \cdots \leq H$.  The underlying set $\mathcal{K}$ here may be a union of several $\mathcal{K}_i$ or simply a single one. Let
\begin{equation}\label{eq:M}
M=M(\mathcal{K})=\max_{(\chi,\chi')\in \mathcal{K}^2}\operatorname{cond}(\overline{\chi} \chi').
\end{equation}
Clearly, $M=O(1)$ if $\mathcal{K}=\mathcal{K}_i$, $i\geq 1$. This would be an advantage when estimating the value of $S$, as we can see below.

\begin{lemma}\label{lem:Density1}
Let $x,y,z,\epsilon,\Lambda>0$ be given and
\begin{equation}\label{eq:k}
k=\begin{cases}
2(\phi+3x+y+z), &M=O(1),\\
2(2\phi+3x+y+z), &\text{otherwise}.
\end{cases}
\end{equation}
Then for $q\geq q(k,\epsilon)$ we have
\begin{equation}\label{eq:Density1Result}
\sum_{j\geq 1} e^{-k\max\{\lambda_{j},\Lambda\}}\leq (1+\epsilon) C(x,y,z,\Lambda,\lambda_0),
\end{equation}
where
\begin{equation}\label{eq:C}
C(x,y,z,\Lambda,\lambda_0)=\frac{1}{xz}\bb{\frac{1}{2}+\frac{\phi+x}{y}}\sqrt{B\bb{\phi+x+y,\Lambda,\lambda_0}B(z,\Lambda,\lambda_0)}.
\end{equation}
\end{lemma}

\begin{proof}
Let $w(\rho)=w(\lambda )$ be a positive-valued function to be determined later. When $q$ is sufficiently large, all character can be assumed to be non-principal, since Riemann zeta function has no zeroes in \eqref{eq:RestrictedArea}. By following the argument of \cite[Page 19-25]{Pintz2018}, with a slight change on the notations of parameters, it is easy to show that for $q\geq q(k,\epsilon)$,
\begin{equation}\label{eq:KeyInequality}
\begin{split}
&\bb{\sum_{\rho \in \mathcal{Z}(\mathcal{K},H)}w(\rho)}^2\leq \frac{1+\epsilon}{xz}\bb{\frac{1}{2}+\frac{\phi+x}{y}} \\
&\sum_{\chi\in \mathcal{K}}\sum_{\rho \in \mathcal{Z}(\chi,H)}\sum_{\rho' \in \mathcal{Z}(\chi,H)}w(\rho) w(\rho') e^{k\frac{\lambda+\lambda'}{2}}\sqrt{(G_1G_2)(\lambda+\lambda'-i(\mu-\mu'))}
\end{split},
\end{equation}
where $G_1$ is defined by $F_2$ using $t_0=t_1:=x+y+\phi$ and $G_2$ using $t_0=t_2:=z$.
Now by the basic inequality $2ab\leq a^2+b^2$ and Cauchy-Schwartz inequality, for any $w_0(\rho)>0$ and $w(\rho)=e^{-k\lambda}w_0(\rho)$ we derive from \eqref{eq:lem2Result} that
\begin{equation}\label{eq:CauchySchwartz}
\begin{split}
&\sum_{\chi\in \mathcal{K}}\sum_{\rho \in \mathcal{Z}(\chi,H)}\sum_{\rho' \in \mathcal{Z}(\chi,H)}w(\rho) w(\rho') e^{k\frac{\lambda+\lambda'}{2}}\sqrt{(G_1G_2)(\lambda+\lambda'-i(\mu-\mu'))} \\
& \leq \sum_{\chi\in \mathcal{K}}\sum_{\rho \in \mathcal{Z}(\chi,H)}w(\rho)^2 e^{k\lambda}\sum_{\rho' \in \mathcal{Z}(\chi,H)}\sqrt{(G_1G_2)(\lambda+\lambda'-i(\mu-\mu'))} \\
&= \sum_{\chi\in \mathcal{K}}\sum_{\rho \in \mathcal{Z}(\chi,H)}w(\rho)w_0(\rho) \sum_{\rho' \in \mathcal{Z}(\chi,H)}\sqrt{(G_1G_2)(\lambda+\lambda'-i(\mu-\mu'))} \\
& \leq \max_{\rho}w_0(\rho)\sqrt{B(t_1,\lambda,\lambda_0) B(t_2,\lambda,\lambda_0)}\sum_{\rho\in \mathcal{Z}(\mathcal{K},H)}w(\rho).
\end{split}
\end{equation}
This means
\begin{equation}\label{eq:Simplified}
\begin{split}
\sum_{\rho \in \mathcal{Z}(\mathcal{K},H)}w(\rho)\leq \frac{1+\epsilon}{xz}\bb{\frac{1}{2}+\frac{\phi+x}{y}}\max_{\rho}w_0(\rho) \sqrt{B(t_1,\lambda,\lambda_0)B(t_2,\lambda,\lambda_0)}.
\end{split}
\end{equation}
Let $w_0(\rho)=e^{-\frac{(t_1+t_2)\max\{\Lambda-\lambda,0\}}{2}}$. By Lemma \ref{lem:3}, we have
\begin{equation}\label{eq:w0Bound}
\begin{split}
&w_0(\rho) \sqrt{B(t_1,\lambda,\lambda_0)B(t_2,\lambda,\lambda_0)}\leq \sqrt{B(t_1,\Lambda,\lambda_0)B(t_2,\Lambda,\lambda_0)}
\end{split}
\end{equation}
This implies
\begin{equation}\label{eq:FinalBound}
\sum_{\rho \in \mathcal{Z}(\mathcal{K},H)}w(\rho)\leq (1+\epsilon)C(x,y,z,\Lambda,\lambda_0).
\end{equation}
Note that when $\lambda\leq \Lambda$,
\begin{equation}\label{eq:wLower}
\begin{split}
w(\rho)&=e^{-k\lambda -\frac{t_1+t_2}{2}\max\{\Lambda-\lambda,0\}} \\
&\geq \min \{ e^{-\frac{t_1+t_2}{2}\Lambda}, e^{-k\lambda} \} \\
&\geq \min \{ e^{-2(3x+y+z+\phi)}, e^{-k\lambda} \} \\
&\geq e^{-k\Lambda}
\end{split}
\end{equation}
which concludes the proof.
\end{proof}

From this result it is easy to derive estimates for $N(\Lambda)$ (in this section $N(\Lambda)$ is restricted to the set $\mathcal{K}$), that is,
\begin{equation}\label{eq:NEstimate}
N(\Lambda)\leq C(x,y,z,\Lambda,\lambda_0) e^{k\Lambda}.
\end{equation}
For example, by $\lambda_0=1/2$, $x=1/12$, $y=1/4$ and $z=1/6$ we have
\begin{equation}\label{eq:Example}
N(\Lambda)\leq 19.62e^{\frac{8}{3}\Lambda},\quad \Lambda\geq 1.311.
\end{equation}
This improves the corresponding result of \cite{Pintz2018}, where the coefficient is $22.281$. To make full use of this estimate, for any $\Lambda\in [c_0,H]$ we separate $S$ into two parts, that is,
\begin{equation}\label{eq:TandR}
T=T(\Lambda)=\sum_{j\geq 1} e^{-\frac{10}{3} \max \{\lambda_{j},\Lambda\}},\quad R=R(\Lambda)=S-T(\Lambda).
\end{equation}
Then for $q\geq q(k,\epsilon)$ and $k\in [0,10/3]$, we have by Lemma \ref{lem:Density1} that
\begin{equation}\label{eq:TEstimate}
T(\Lambda)\leq (1+\epsilon)e^{-(\frac{10}{3}-k)\Lambda} C(x,y,z,\Lambda,\lambda_0).
\end{equation}
Choosing optimal parameters for each given $\Lambda$ and $\lambda_0$ will give desired bounds of $T(\Lambda)$ in the general case.  When $M=O(1)$, $T(\Lambda)$ can be estimated by a more efficient way. In this case we will only apply \eqref{eq:TEstimate} with $\lambda_0=0$ and $\Lambda\geq 5.2$ to obtain the following result.

\begin{corollary}\label{cor:1}
Let $q$ be sufficiently large and $M=O(1)$. For $\Lambda\geq 5.2$, we have
\begin{equation}\label{eq:CorollaryResult}
T(\Lambda)\leq 100 e^{-2.22\Lambda}\leq 0.001.
\end{equation}
\end{corollary}

\begin{proof}
Since $C(x,y,z,\Lambda,0)$ decreases as $\Lambda$ increases, it follows that for $\Lambda\geq 5.2$,
\begin{equation}\label{eq:TLambda}
T(\Lambda)\leq (1+\epsilon)e^{-(\frac{10}{3}-k)\Lambda} C(x,y,z,5.2,0).
\end{equation}
When $q$ is sufficiently large, the choice
\begin{equation}\label{eq:Choice}
x = 0.029,\quad y = 0.083,\quad z = 0.052
\end{equation}
gives
\begin{equation}\label{eq:Values}
C(x,y,z,5.2,0)\leq 99.728\cdots,\quad \frac{10}{3}-k= 2.226\cdots,
\end{equation}
which immediately implies \eqref{eq:CorollaryResult}. Meanwhile, we see
\begin{equation}\label{eq:T52}
T(5.2)\leq 100e^{-2.22\cdot 5.2}=0.00096\cdots.
\end{equation}
\end{proof}

When $M=O(1)$ and $\Lambda\leq 5.2$, the above estimate shows that the terms for which $\lambda_j\geq 5.2$ in $T(\Lambda)$ are negligible. As for the terms $\lambda_j\in [\Lambda, 5.2]$, we can bound them by the new zero density estimate of Heath-Brown, which is generalized to our setting by Pintz.

Let $G(z)$ be the Laplace transform of
\begin{equation}\label{eq:gFunction}
g(u)=\frac{1}{30}(2-u)^3(4+6u+u^2),\quad u\in [0,2],
\end{equation}
that is,
\begin{equation}\label{eq:G}
G(z)=\int_0^2g(u)e^{-uz}du.
\end{equation}
For $x>0$, let $f_x(u):=x\cdot g(ux)$ whose Laplace transform is denoted by $F_x(z)=G(z/x)$. Define for $j\geq 1$,
$$\psi_{j}=\frac{F_x(\lambda_{j}-\lambda_{0})}{F_x(-\lambda_{0})}, \quad \psi=\frac{F_x(\Lambda-\lambda_{0})}{F_x(-\lambda_{0})},\quad \xi=\frac{\phi}{2}\frac{f_x(0)}{F_x(-\lambda_{0})},\quad \Delta=\psi-\xi.$$
Let
\begin{equation}\label{eq:D}
D=D(\Lambda)=\sum_{\substack{j\geq 1\\ \lambda_j\leq \Lambda}} (\psi_{j}-\psi).
\end{equation}
The following two basic inequalities based on these quantities are essentially established in  \cite{Pintz2018}.

\begin{lemma}\label{lem:Density2}
With the notations above and any $\epsilon>0$, we have for $q\geq q(\epsilon)$ that
\begin{equation}\label{eq:Basic1}
(\Delta^2-\xi-\epsilon)N+ 2\Delta D \leq 1-\xi
\end{equation}
provided $x\geq 4\lambda_{0}/5$ and $\Delta\geq \sqrt{\xi+\epsilon}$. When $M=O(1)$, we have instead
\begin{equation}\label{eq:Basic2}
(\Delta^2-\epsilon)N+2\Delta D \leq 1
\end{equation}
provided $x\geq 4\lambda_0/5$ and $\Delta\geq \sqrt{\epsilon}$.
\end{lemma}

\begin{proof}
The inequality \eqref{eq:Basic1} is exactly \cite[Equation (8.3)]{Pintz2018}. The inequality \eqref{eq:Basic2} can be derived in a similar way from \cite[Equation (7.29)]{Pintz2018}.
\end{proof}

We first utilize this to continue the estimate of $T(\Lambda)$ provided $M=O(1)$. Note that by \eqref{eq:Basic2} we have
\begin{equation}\label{eq:ZeroDensity2}
N\leq \frac{1}{\Delta^2-\epsilon},
\end{equation}
which is exactly \cite[Theorem I]{Pintz2018}.
Let $\Lambda=\Lambda_0\leq \Lambda_1\leq \cdots \leq \Lambda_{200}=5.2$ be equally distributed. For each $\Lambda_i$, we apply \eqref{eq:ZeroDensity2} to obtain a $N_i$ with $x=x_i$ being chosen optimally. Recall that $T(5.2)\leq 0.001$, it follows
\begin{equation}\label{eq:TSum}
T(\Lambda)\leq e^{-\frac{10}{3} \Lambda_0} N_0+\sum_{i=1}^{200}(N_{i}-N_{i-1})e^{-\frac{10}{3} \Lambda_{i-1}}+0.001.
\end{equation}

Until now we have introduced all methods estimating $T(\Lambda)$. As for the estimate of
\begin{equation}\label{eq:R}
R(\Lambda)=\sum_{\substack{j\geq 1\\\lambda_{j}\leq \Lambda}}(e^{-\frac{10}{3}\lambda_{j}}-e^{-\frac{10}{3}\Lambda}),
\end{equation}
we will basically follow the treatment of \cite{Pintz2018} to relate it to the sum
\begin{equation}\label{eq:D2}
D(\Lambda)= \sum_{\substack{j\geq 1\\\lambda_j\leq \Lambda}}(\psi_{j}-\psi).
\end{equation}
In the sequel we will always assume that $N(\Lambda)\geq 3$, otherwise bounds from trivial methods would be better. Meanwhile, we will preassign the values of $\lambda_1$, $\lambda_2$, and a lower bound $\lambda_3\geq \lambda^*$. Hence $\lambda_0=\lambda_1$ and we can write
\begin{equation}\label{eq:RDecomposition}
R(\Lambda)=\sum_{j=1}^2(e^{-\frac{10}{3}\lambda_{j}}-e^{-\frac{10}{3}\Lambda})+\sum_{\substack{j\geq 3\\\lambda_{j}\leq \Lambda}}(e^{-\frac{10}{3}\lambda_{j}}-e^{-\frac{10}{3}\Lambda}).
\end{equation}
Since the function
\begin{equation}\label{eq:Function}
\frac{e^{au}-1}{e^{bu}-1},\quad (b\geq a),
\end{equation}
decreases as $u$ increases, when $x\geq 3/5$ we see that
\begin{equation}\label{eq:Ratio}
\frac{\psi_j-\psi}{e^{K(\Lambda-\lambda_j)}-1}=\frac{1}{F_x(-\lambda_0)}\int_0^{2/x}F_x(u) e^{-u(\Lambda-\lambda_0)}\frac{e^{u(\Lambda-\lambda_j)}-1}{e^{\frac{10}{3}(\Lambda-\lambda_j)}-1}du
\end{equation}
increases as $\lambda_j$ increases. Hence by our labelling of $\lambda_j$,
\begin{equation}\label{eq:SumBound2}
\sum_{\substack{j\geq 3\\\lambda_{j}\leq \Lambda}}(e^{-\frac{10}{3}\lambda_{j}}-e^{-\frac{10}{3}\Lambda})\leq \frac{e^{-\frac{10}{3}\lambda^*}-e^{-\frac{10}{3}\Lambda}}{\psi^*-\psi}\sum_{\substack{j\geq 3\\\lambda_{j}\leq \Lambda}}(\psi_j-\psi),
\end{equation}
where
\begin{equation}\label{eq:psiStar}
\psi^*=\frac{F_x(\lambda^*-\lambda_0)}{F_x(-\lambda_0)}.
\end{equation}
Recall that by Lemma \ref{lem:Density2} we have
\begin{equation}\label{eq:Inequality}
N (\Delta^2-\xi-\epsilon)+ 2\Delta D \leq 1-\xi.
\end{equation}
To make full use of this inequality, for any given $\lambda_0$ we will choose $\Lambda$ to be as large as possible such that there exists $x\geq \max \{4\lambda_0/5, 3/5\}$ for which $\Delta^2>\xi$. Hence $\Delta^2\approx \xi$ and with minor loss we obtain
\begin{equation}\label{eq:DBound}
D\leq \frac{1-\xi}{2\Delta},
\end{equation}
which is exactly \cite[Theorem J]{Pintz2018}. The desired pairs of $\lambda_0$ and $\Lambda$ in our argument are
\begin{equation}\label{eq:Pairs}
\begin{cases}
\lambda_0=0.60,&\Lambda=1.348, \\
\lambda_0=0.62,& \Lambda= 1.355, \\
\lambda_0=0.64,& \Lambda= 1.363, \\
\lambda_0=0.66,& \Lambda= 1.370, \\
\lambda_0=0.68,& \Lambda= 1.378, \\
\lambda_0=0.92,& \Lambda= 1.467.
\end{cases}
\end{equation}
Gathering these inequalities, we finally arrive at
\begin{equation}\label{eq:RFinal}
R(\Lambda)\leq \sum_{j=1}^2(e^{-\frac{10}{3}\lambda_{j}}-e^{-\frac{10}{3}\Lambda})+\frac{e^{-\frac{10}{3}\lambda^*}-e^{-\frac{10}{3}\Lambda}}{\psi^*-\psi} \bb{\frac{1-\xi}{2\Delta}-\sum_{j=1}^2(\psi_j-\psi)},
\end{equation}
provided $x\geq \max\{4\lambda_0/5,3/5\}$, $\Delta>\sqrt{\xi}$ and $q$ sufficiently large. Since we require $x\geq 3/5$, by \cite[Theorem K]{Pintz2018} the above bound is still valid when $\lambda_1$ and $\lambda_2$ get larger than the preassigned values.

The next refinement comes when estimating $R(\Lambda)$ for $M=O(1)$. Recall in this case we have
\begin{equation}\label{eq:Basic2M}
(\Delta^2-\epsilon)N+2\Delta D\leq 1.
\end{equation}
If we argue as before, it would follow
\begin{equation}\label{eq:DBoundM}
D\leq \frac{1}{2\Delta},
\end{equation}
which is actually worse than \eqref{eq:DBound} since we lose a $-\xi$ in the numerator. The key observation is that, by our setting of $\Lambda$, we have $\Delta^2\approx \xi$ and thus an additional assumption $N\geq N^*+1$ will imply, approximately, a better bound
\begin{equation}\label{eq:DBoundM2}
D\leq \frac{1-(N^*+1)\xi}{2\Delta}.
\end{equation}
The empirical choice of $N^*$ will be $3$ or $4$. In this way, we have either $N\leq N^*$ which means
\begin{equation}\label{eq:RCase1}
R(\Lambda)\leq \sum_{j=1}^2(e^{-\frac{10}{3}\lambda_{j}}-e^{-\frac{10}{3}\Lambda})+(N^*-2)(e^{-\frac{10}{3}\lambda^*}-e^{-\frac{10}{3}\Lambda})
\end{equation}
or
\begin{equation}\label{eq:RCase2}
R(\Lambda)\leq \sum_{j=1}^2(e^{-\frac{10}{3}\lambda_{j}}-e^{-\frac{10}{3}\Lambda})+\frac{e^{-\frac{10}{3}\lambda^*}-e^{-\frac{10}{3}\Lambda}}{\psi^*-\psi} \bb{\frac{1-(N^*+1)\Delta^2}{2\Delta}-\sum_{j=1}^2(\psi_j-\psi)}.
\end{equation}
Also, this bound is still valid for larger $\lambda_1$ and $\lambda_2$. Now we have introduced all methods estimating $R(\Lambda)$.

It remains to apply these estimates to prove Theorem \ref{thm:main2}.

\section{Proof of theorem \ref{thm:main2}}\label{sec:ProofTheorem2}

We now roughly distinguish the following five cases
\begin{itemize}
\item $\lambda_{1,1}\in [c_0,0.01]$
\item $\lambda_{1,1}\in [0.01,0.40]$
\item $\lambda_{1,1}\in [0.40,0.60]$
\item $\lambda_{1,1}\in [0.60,0.62]$
\item $\lambda_{1,1}\in [0.62,H]$
\end{itemize}
In different cases, we will choose different values of $\Lambda$ to separate $S_i=T_i+R_i$ to apply those estimates.

\textbf{The case $\lambda_{1,1}\in [c_0,0.01]$:}
In this case we let $\Lambda=\max \{5.68,1.09\log \lambda_{1,1}^{-1}\}$. Hence by Lemma \ref{lem:Cases} we have $N(\Lambda)\leq 1$, which means
\begin{equation}\label{eq:RiCase1}
R_i=\begin{cases}
e^{-\frac{10}{3}\lambda_{1,1}}-e^{-\frac{10}{3}\Lambda},& i=1, \\
0, &i\geq 2.
\end{cases}
\end{equation}
It follows that
\begin{equation}\label{eq:SumSquaresCase1}
\begin{split}
\sum_{i\geq 1}S_i^2&\leq e^{-2\cdot \frac{10}{3}\cdot \lambda_{1,1}}+2T_1+\sum_{i\geq 1} T_i^2 \\
&\leq e^{-\frac{20}{3}\lambda_{1,1}}+\max_{i\geq 1}T_i \bb{2+\sum_{i\geq 1}T_i}.
\end{split}
\end{equation}
Since $\Lambda\geq 5.2$, by Corollary \ref{cor:1} we have
\begin{equation}\label{eq:MaxTi}
\max_{i\geq 1}T_i\leq 100e^{-2.22\Lambda}.
\end{equation}
Meanwhile using \eqref{eq:TEstimate} with $\lambda_0=0$ and $\Lambda=5.68$ we see
\begin{equation}\label{eq:SumTi}
\sum_{i\geq 1}T_i\leq 0.02.
\end{equation}
Thus
\begin{equation}\label{eq:FinalBoundCase1}
\sum_{i\geq 1}S_i^2\leq e^{-\frac{20}{3}\lambda_{1,1}}+202e^{-2.22 \Lambda}.
\end{equation}
The critical value $\lambda_{1,1}=e^{-\frac{5.68}{1.09}}$ gives the bound $0.965$. Noticing that the right hand side is strictly convex if $\lambda_{1,1}\leq e^{-\frac{5.68}{1.09}}$ and decreasing if $\lambda_{1,1}\geq e^{-\frac{5.68}{1.09}}$, we conclude that there exists $c_1>0$ depending only on $c_0$ such that
\begin{equation}\label{eq:FinalResultCase1}
\sum_{i\geq 1}S_i^2\leq 1-c_1
\end{equation}
as desired.

\textbf{The case $\lambda_{1,1}\in [0.01,0.40]$:}
We now further distinguish three cases: $\lambda_{1,1}\in [0.01,0.10]$, $\lambda_{1,1}\in [0.10,0.30]$ and $\lambda_{1,1}\in [0.30,0.40]$. However, they will be treated in a similar manner. Since our method takes its advantages when $\lambda_0$ is large, it suffices to treat the special case $\lambda_{1,1}\in [0.30,0.40]$.

Now by Lemma \ref{lem:Cases} we have $N(1.29)=1$. Hence we may let $\Lambda=1.29$ to obtain
\begin{equation}\label{eq:SumSquaresCase2}
\sum_{i\geq 1}S_i^2\leq (e^{-\frac{10}{3}\cdot 0.3}+T_1)^2+\sum_{i\geq 2} T_i^2.
\end{equation}
Using \eqref{eq:TEstimate} and \eqref{eq:NEstimate}, with $\lambda_0=0.30$ for $i=1$ and $\lambda_0=1.29$ for $i\geq 2$, we see
\begin{equation}\label{eq:TiValues}
T_i\leq \begin{cases}
0.1231,& i=1, \\
0.0971,& i\geq 2,
\end{cases}
\quad \sum_{i\geq 2}T_i\leq 6.85,
\end{equation}
which means
\begin{equation}\label{eq:FinalBoundCase2}
\sum_{i\geq }S_i^2\leq (0.3679+0.1231)^2+0.0971\cdot 6.85=0.9062\cdots.
\end{equation}

\textbf{The case $\lambda_{1,1}\in [0.40,0.60]$:}
From now on estimates of $R(\Lambda)$ will get involved. By Lemma \ref{lem:Cases}, we have $N(0.92)\leq 2$. So there exists $i_0\geq 1$ such that $\lambda_{i,1}\geq 0.92$ holds for $i>i_0$. Hence we may first consider
\begin{equation}\label{eq:SumLargei}
\sum_{i>i_0}S_i^2.
\end{equation}
Let $\Lambda=1.467$, then we have
\begin{equation}\label{eq:DecompositionLarge}
\begin{split}
\sum_{i>i_0}S_i^2&=\sum_{i>i_0}(T_i+R_i)^2 \\
&=\sum_{i>i_0}R_i^2+\sum_{i\geq i_0}T_i\bb{2R_i+T_i}.
\end{split}
\end{equation}
Using \eqref{eq:TEstimate} and \eqref{eq:NEstimate}, with $\lambda_0=0.92$ and $\Lambda=1.467$, we see
\begin{equation}\label{eq:TiLarge}
\max_{i>i_0}T_i\leq 0.0704,\quad \sum_{i>i_0}T_i\leq 6.33.
\end{equation}
Meanwhile, using \eqref{eq:RCase1}, \eqref{eq:RCase2} and \eqref{eq:DBoundM2} with $\lambda_1=\lambda_2=\lambda^*=0.92$ and $N^*=4$, we have
\begin{equation}\label{eq:RiLarge}
\max_{i>i_0} R_i\leq 0.157,\quad \sum_{i>i_0} R_i\leq 0.457.
\end{equation}
Hence
\begin{equation}\label{eq:FinalLarge}
\sum_{i>i_0}S_i^2\leq 0.157\cdot 0.557+0.0704\cdot (2\cdot 0.457+6.33)=0.5817\cdots.
\end{equation}
On the other hand, for $i\leq i_0$ we always have the bound
\begin{equation}\label{eq:TiSmall}
T_i\leq 0.0808
\end{equation}
from \eqref{eq:TEstimate} with $\lambda_0=0.40$ and $\Lambda=1.467.$ But for $R_i$, the case $i_0=2$ gives $\lambda_1=0.40$ and $\lambda_2=\lambda^*=0.92$. So by \eqref{eq:RCase1} and \eqref{eq:RCase2} with $N^*=4$ we have
\begin{equation}\label{eq:RiSmall1}
R_1,R_2\leq 0.374,
\end{equation}
and hence
\begin{equation}\label{eq:FinalSmall1}
\sum_{i\leq i_0}S_i^2\leq 2(0.374+0.0808)^2=0.4136\cdots.
\end{equation}
The case $i_0=1$ gives $\lambda_1=\lambda_2=0.40$ and $\lambda^*=0.92$, which means
\begin{equation}\label{eq:RiSmall2}
R_1\leq 0.556
\end{equation}
and thus
\begin{equation}\label{eq:FinalSmall2}
\sum_{i\leq i_0}S_i^2\leq (0.556+0.0808)^2=0.4055\cdots.
\end{equation}
In conclusion, we have
\begin{equation}\label{eq:FinalCase3}
\sum_{i\geq 1} S_i^2\leq 0.995\cdots.
\end{equation}

\textbf{The case $\lambda_{1,1}\in [0.60,H]$:}
We further distinguish four cases: $\lambda_{1,1}\in [0.60,0.62]$, $\lambda_{1,1}\in [0.62,0.64]$, $\lambda_{1,1}\in [0.64,0.68]$ and $\lambda_{1,1}\in [0.68,H]$.

When $\lambda_{1,1}\in [0.60,0.62]$, let $\Lambda=1.348$. By Lemma \ref{lem:Cases}, in this case we have either $N(0.85)\leq 1$ or $N(0.91)\leq 2$. When $N(0.91)\leq 2$, using \eqref{eq:TEstimate} and \eqref{eq:NEstimate} with $\lambda_0=0.60$ and $\Lambda=1.348$ we have
\begin{equation}\label{eq:TiCase4}
\max_{i\geq 1} T_i\leq 0.101,\quad \sum_{i\geq 1}T_i\leq 7.63.
\end{equation}
And using \eqref{eq:TEstimate} with $\lambda_0=0.91$ we have
\begin{equation}\label{eq:TiLargeCase4}
\max_{i\geq 3} T_i\leq 0.0927.
\end{equation}
Meanwhile, using \eqref{eq:RCase1}, \eqref{eq:RCase2} and \eqref{eq:DBoundM2} with $\lambda_1=\lambda_2=0.60$, $\lambda^*=0.91$ and $N^*=3$, we obtain
\begin{equation}\label{eq:RiCase4}
\max_{i\geq 1} R_i\leq 0.300,\quad \sum_{i\geq 1}R_i\leq 0.581.
\end{equation}
Hence
\begin{equation}\label{eq:FinalCase4}
\begin{split}
\sum_{i\geq 1}S_i^2=&\sum_{i\geq 1}R_i^2+\sum_{i\geq1} T_i(2R_i+T_i) \\
\leq &0.3\cdot 0.581+0.0927\cdot (2\cdot 0.581+7.63) \\
&+2\cdot (0.1001-0.0927)(2\cdot 0.3+0.1001) \\
=&0.9996\cdots.
\end{split}
\end{equation}
The case $N(0.85)\leq 1$ is a simple analogue to this. We can repeat this argument for $\lambda_{1,1}\in [0.62,0.64]$. But when $\lambda_{1,1}\in [0.64,0.66]$, $\lambda_{1,1}\in [0.66,0.68]$ and $\lambda_{1,1}\in [0.68,H]$, we may use the fact
\begin{equation}\label{eq:MaxTiLarge}
\max_{i\geq 5}T_i
\end{equation}
has better estimate due to Lemma \ref{lem:Cases}.

In all cases, we have shown that 
\beq
 \sum_{i\geq 1}S_i^2\leq 1-c_1
\eeq for some $c_1>0$ depending only on $c_0$, which completes the proof of Theorem \ref{thm:main2}.

\end{document}